\documentclass[11pt]{article}
\usepackage[margin=1in]{geometry} 
\usepackage{amssymb,amsfonts,amsmath,bbm,mathrsfs,stmaryrd}
\usepackage{xcolor}
\usepackage{url}

\usepackage[colorlinks,
             linkcolor=black!75!red,
             citecolor=blue,
             pdftitle={Top_gen},
             pdfauthor={Alexandru Chirvasitu, Michael Brannan},
             pdfproducer={pdfLaTeX},
             pdfpagemode=None,
             bookmarksopen=true
             bookmarksnumbered=true]{hyperref}

\usepackage{tikz}
\usetikzlibrary{arrows,calc,decorations.pathreplacing,decorations.markings,intersections,shapes.geometric,through,fit,shapes.symbols,positioning,decorations.pathmorphing}

\usepackage{braket}

\usepackage[amsmath,thmmarks,hyperref]{ntheorem}
\usepackage{cleveref}

\crefname{section}{Section}{Sections}
\crefformat{section}{#2Section~#1#3} 
\Crefformat{section}{#2Section~#1#3} 

\crefname{subsection}{\S}{\S\S}
\crefformat{subsection}{#2\S#1#3} 
\Crefformat{subsection}{#2\S#1#3} 

\theoremstyle{plain}

\newtheorem{lemma}{Lemma}[section]
\newtheorem{proposition}[lemma]{Proposition}
\newtheorem{corollary}[lemma]{Corollary}
\newtheorem{theorem}[lemma]{Theorem}
\newtheorem{conjecture}[lemma]{Conjecture}

\theoremstyle{nonumberplain}
\newtheorem{theoremN}{Theorem}

\theoremstyle{plain}
\theorembodyfont{\upshape}
\theoremsymbol{\ensuremath{\blacklozenge}}

\newtheorem{definition}[lemma]{Definition}
\newtheorem{example}[lemma]{Example}
\newtheorem{remark}[lemma]{Remark}

\crefname{definition}{definition}{definitions}
\crefformat{definition}{#2definition~#1#3} 
\Crefformat{definition}{#2Definition~#1#3} 

\crefname{ex}{example}{examples}
\crefformat{example}{#2example~#1#3} 
\Crefformat{example}{#2Example~#1#3} 

\crefname{remark}{remark}{remarks}
\crefformat{remark}{#2remark~#1#3} 
\Crefformat{remark}{#2Remark~#1#3} 

\crefname{convention}{convention}{conventions}
\crefformat{convention}{#2convention~#1#3} 
\Crefformat{convention}{#2Convention~#1#3}

\crefname{lemma}{lemma}{lemmas}
\crefformat{lemma}{#2lemma~#1#3} 
\Crefformat{lemma}{#2Lemma~#1#3} 

\crefname{proposition}{proposition}{propositions}
\crefformat{proposition}{#2proposition~#1#3} 
\Crefformat{proposition}{#2Proposition~#1#3} 

\crefname{corollary}{corollary}{corollaries}
\crefformat{corollary}{#2corollary~#1#3} 
\Crefformat{corollary}{#2Corollary~#1#3} 

\crefname{theorem}{theorem}{theorems}
\crefformat{theorem}{#2theorem~#1#3} 
\Crefformat{theorem}{#2Theorem~#1#3} 

\crefname{assumption}{assumption}{Assumptions}
\crefformat{assumption}{#2assumption~#1#3} 
\Crefformat{assumption}{#2Assumption~#1#3} 

\crefname{equation}{}{}
\crefformat{equation}{(#2#1#3)} 
\Crefformat{equation}{(#2#1#3)}

\theoremstyle{nonumberplain}
\theoremsymbol{\ensuremath{\blacksquare}}

\newtheorem{proof}{Proof}
\newcommand\pf[1]{\newtheorem{#1}{Proof of \Cref{#1}}}

\newcommand\bC{\mathbb C}

\newcommand\bG{\mathbb G}
\newcommand\bH{\mathbb H}

\newcommand\bN{\mathbb N}

\newcommand\bS{\mathbb S}

\newcommand\bZ{\mathbb Z}

\newcommand\cA{\mathcal A}
\newcommand\cB{\mathcal B}
\newcommand\cC{\mathcal C}
\newcommand\cD{\mathcal D}

\newcommand\cL{\mathcal L}
\newcommand\cM{\mathcal M}

\DeclareMathOperator{\id}{id}

\newcommand{\qedhere}{\mbox{}\hfill\ensuremath{\blacksquare}}

\title{Topological generation and matrix models for quantum reflection groups }
\author{Michael Brannan, Alexandru Chirvasitu, Amaury Freslon}

\begin{document}

\date{}

\newcommand{\Addresses}{{
  \bigskip
  \footnotesize

  \textsc{Michael Brannan, Department of Mathematics, Texas A\&M University, College Station,
    TX 77843-3368, USA}\par\nopagebreak \textit{E-mail address}:
  \texttt{mbrannan@math.tamu.edu}
  
  \medskip
  
  \textsc{Alexandru Chirvasitu, Department of Mathematics, University at Buffalo, Buffalo,
    NY 14260-2900, USA}\par\nopagebreak \textit{E-mail address}:
  \texttt{achirvas@buffalo.edu}

  \medskip
  
  \textsc{Amaury Freslon, Laboratoire de Math\'ematiques d’Orsay, Univ.  Paris-Sud, CNRS, Universit\'e Paris-Saclay, 91405 Orsay, France}\par\nopagebreak
  \textit{E-mail address}: \texttt{amaury.freslon@math.u-psud.fr}
}}

\maketitle

\begin{abstract}
We establish several new topological generation results for the quantum permutation groups $S^+_N$ and the quantum reflection groups $H^{s+}_N$.  We use these results to show that these quantum groups admit sufficiently many ``matrix models''.  In  particular, all of these quantum groups have residually finite discrete duals (and are, in particular, hyperlinear), and certain ``flat'' matrix models for $S_N^+$ are inner faithful. 
\end{abstract}


\vspace{.5cm}


\tableofcontents

\section{Introduction}

The central objects of study in this paper are {\it quantum permutations} and their generalizations, {\it quantum reflections}.  Given $N \in \bN$ and an Hilbert space $H$, an $N\times N$ matrix $P = [P_{ij}]_{1 \le i,j \le N} \in M_N(B(H))$ ($B(H)$ being the C$^\ast$-algebra of bounded linear operators on $H$) is called a quantum permutation matrix (or {\it magic unitary}) if its entries $P_{ij}$ are self-adjoint projections satisfying the relations $\sum_i P_{ij} = 1_{B(H)} = \sum_{j} P_{ij}$ for each $1 \le i,j \le N$.  

The simplest examples of quantum permutation matrices are of course the classical permutation matrices (which correspond to those quantum permutations $P$ associated to a one-dimensional Hilbert space $H$).  In fact, more generally any quantum permutation $P$ with commuting entries $\{P_{ij}\}_{i,j}$ corresponds to a subset $X \subseteq S_N$ of permutation matrices.  Indeed, in this case the commutative C$^\ast$-algebra $C^*(\{P_{ij}\}_{i,j})$ generated by the $P_{ij}$'s is by Gelfand duality isomorphic to $C(X)$, the C$^\ast$-algebra of complex functions on some subset $X \subseteq S_N$.  In particular $P$ is identified this way with the identity function on $X \subset S_N \subset M_N(\bC)$.  On the other hand, if one now considers quantum permutations $P$ whose entries {\it do not} commute, the structure of these objects becomes much less well-understood.  From an operator algebraic point of view, this should come as no surprise, as the C$^\ast$-algebras $C^*(P_{ij} \ | \ 1 \le i,j \le N) \subset B(H)$ generated by the entries of a quantum permutation matrix $P$ with non-commuting entries can be highly non-trivial (e.g., can contain the free group C$^\ast$-algebras as quotients \cite{Wan98}).  Nonetheless, such ``genuinely quantum'' quantum permutations arise naturally in a variety of contexts.   For example, in quantum information theory, quantum permutation matrices arise naturally in the framework of non-local games and go under the name ``projective permutation matrices'' \cite{Ats16, Mu17, Lu17, Lu18}.  From the perspective of non-commutative geometry and quantum group theory, $N \times N$ quantum permutation matrices were discovered by Wang \cite{Wan98} to be precisely the structure that encode the {\it quantum symmetries} of a finite set of $N$ points.  More precisely, Wang considered the universal unital C$^\ast$-algebra \[A = C^*\Big(u_{ij}, 1 \le i,j \le N \ \big| \ u_{ij} = u_{ij}^* = u_{ij}^2 \ \& \ \sum_{i} u_{ij} = \sum_j u_{ij} = 1 \ \forall i,j\Big),\]
generated by the coefficients of a ``universal'' $N \times N$ quantum permutation matrix.  Wang then showed that there exists a {\it compact quantum group}, $S_N^+$, acting universally and faithfully on the set of  $N$ points in such a way that $A$ gets identified with the C$^\ast$-algebra $C(S_N^+)$ of ``continuous functions'' on the ``quantum space'' $S_N^+$.  The quantum group $S_N^+$ is called the {\it quantum permutation group} or {\it quantum symmetry group of $N$ points}.  In contrast to its classical counterpart, $S_N^+$ (or more precisely $A = C(S_N^+)$) is a highly non-commutative and infinite-dimensional object.  It is one of our main goals in this paper is to investigate to what extent the quantum permutation groups $S_N^+$ (and the quantum reflection groups) can be {\it approximated} by elementary finite-dimensional structures. We elaborate briefly on this now.

C$^\ast$-algebraic compact quantum groups as introduced in \cite{Wor87} form the basis of what by now is a rich theory, developing rapidly in a number of different directions. As indicated above, the perspective we adopt in this paper is that the Hopf (C$^*$-)algebras studied in \cite{Wor87,Wor88,Wor98} play the role of function algebras on a ``compact quantum space'' $\bG$ which is equipped with a group structure, and can equivalently be viewed as the complex group algebras of the Pontryagin dual ``discrete quantum group'' $\Gamma$. To keep matters simple, throughout this introduction we write $\bC \Gamma$ for the group algebra of a discrete quantum group $\Gamma$ (see \Cref{se.prel} below for details). 

One aspect of the theory of discrete quantum groups that presents itself naturally from this point of view is that of {\it approximation properties}. This typically refers to the ``accessibility'' of the quantum group (or its associated algebras) via finite structures of some type. The phrase `finite structure' is purposely vague, and lends itself to a variety of interpretations, for example:
\begin{itemize}
\item The {\it amenability} of a discrete quantum group implies the nuclearity of $C^*(\Gamma)$, the universal C$^*$-completion of $\bC \Gamma$. (I.e., the identity map on $C^*(\Gamma)$ can be point-norm approximated by finite-rank completely positive contractions).  The converse is also true, at least for Kac type discree quantum groups \cite{bt-amnbl}.
\item The {\it Haagerup approximation property} of $\Gamma$ corresponds to the point-norm approximation of the identity map on the reduced C$^\ast$-algebra $C^*_{r}(\Gamma)$ by certain well-behaved $L^2$-compact, contractive completely positive maps \cite{dfsw}. 
\item The {\it weak amenbility} of $\Gamma$ corresponds to the point-norm approximation of the identity map on the reduced C$^\ast$-algebra $C^*_{r}(\Gamma)$ by certain well-behaved  finite rank, uniformly bounded completely bounded maps.
\end{itemize}
The above are only a few scattered examples, as we cannot possibly do justice to the vast literature here. We refer to the survey \cite{br-apprx} and its sources for a more expansive discussion on approximation properties for discrete (in fact locally compact) quantum groups. The above list (and all those considered in \cite{br-apprx}) can be thought of as instances {\it internal} approximation of a quantum group by finite structures, since all the above approximating maps are from a given object to itself.  In (quantum) group theory there are of course approximation properties which have an external flavor in the sense that one approximates a given large object by mapping it into smaller auxilliary objects.  For example,
\begin{itemize}
\item A Kac type discrete quantum group $\Gamma$  is {\it hyperlinear} if its quantum group von Neumann algebra $\mathcal L(\Gamma)$ admits-finite dimensional matrix models relative to its Haar trace \cite{bcv}.
\item A Kac type discrete quantum group $\Gamma$ has the {\it Kirchberg factorization property} if there is a net $\varphi_k:C^*( \Gamma) \to M_{n(k)}(\bC)$ of contractive completely positive maps which are asymptotically trace-norm multiplicative and satisfy $h = \lim_k \text{tr}_{n(k)} \circ \varphi_k$ pointwise, where $h$ is the Haar trace and $\text{tr}_{n(k)}$ is the normalized matrix trace.
\item A (finitely generated) discrete quantum group $\Gamma$ is {\it residually finite} if the points of $\bC \Gamma$ are separated by its finite-dimensional $\ast$-representations \cite{chi-rfd,bbcw}.
\item A stronger form of residual finiteness of interest for us is the existence of a faithful or {\it inner faithful}  matrix model $\pi:\bC \Gamma \to  M_N(C(X))$ for some compact Hausdorff $X$ \cite{bb-inner,bc-pauli,bn-flat,bf-model}; we will have more to say about this concept below.   
\end{itemize}

It is the above list of external approximation properties that we are interested in establishing for (the duals of) Wang's quantum permutation groups $S^+_N$ and their generalizations $H^{s+}_N$ (the so-called {\it quantum reflection groups}  \cite{bv-reflection}). We recall some of the details in the preparatory \Cref{se.prel} below, and for now content ourselves to only remind the reader that $H^{s+}_N$ is a quantum version of the classical subgroup $H^s_N\subset GL_N$ consisting of $N\times N$ monomial matrices whose non-zero entries are $s^{th}$ roots of unity (so in particular $H^1_N\cong S_N$, the symmetric group on $N$ symbols). 

The non-commutative topology of $H^{s+}_N$ as a compact quantum group plays a central role in our study of finiteness and approximation properties for their discrete duals $\widehat{H^{s+}_N}$, and hence the types of results proven in the paper. We elaborate briefly: Let $\bG$ be a compact quantum group and let $\bG_i<\bG$ be a family of closed quantum subgroups of $\bG$. The condition that $\bG_i$ {\it topologically generate} $\bG$ was introduced in \cite[Definition 4]{bcv} for a pair of subgroups, but generalizes readily to arbitrary families.  In that paper, topological generation is used in the same fashion we do here: as a tool for lifting finiteness properties from the duals  $\Gamma_i = \widehat{\bG_i}$ to $ \Gamma = \widehat{\bG}$.  For that reason, we prove a number of topological generation results (\Cref{th.top-gen,th.refl-tg}) that can be summarized as:

\begin{theoremN}
  For all $1\le s\le \infty$ and $N\ge 6$, the quantum reflection group $H^{s+}_N$ is topologically generated by its quantum subgroups $S_N$ and $H^{s+}_{N-1}$. For $s=1$ the result also holds for $N=5$.
  \qedhere
\end{theoremN}

This fits into a recurring pattern of topological generation results for infinite families of compact quantum groups. For instance, \cite[Lemma 3.12]{chi-rfd} says in different terms that for $N\ge 5$ the quantum unitary group $U^+_N$ is topologically generated by its quantum subgroup $\bS^1\times U^+_{N-1}$ (product in the category of compact quantum groups, dual to the {\it co}product $C(\bS^1)*C(U^+_{N-1})$ of C$^*$-algebras) and the classical unitary subgroup $U_N<U^+_N$. As for residual finiteness results, we use these  topological generation results to prove (see \Cref{th.rf,th.hns}):

\begin{theoremN}
For $N\ge 4$ and $1\le s\le \infty$ the discrete duals $\widehat{H^{s+}_N}$ of the quantum reflection groups are residually finite. 
  \qedhere
\end{theoremN}
Here too there are precedents for other families: \cite{chi-rfd} treats the discrete duals of free unitary and orthogonal quantum groups. 

We regard the above theorem as one of the main results of the paper, but it has a number of powerful consequences, including the hyperlinearity and Kirchberg factorization property for the selfsame discrete quantum groups \cite{bbcw}, as well as improved estimates for the free entropy dimension of the generators of the associated von Neumann algebras $L^\infty(H^{s+}_N)$ - see Section \ref{section:FED}. 

At this point it is worth highlighting that although our strategy for proving residual finiteness results for the quantum groups $H_N^{s+}$ (by means of inductive topological generation methods) is the same as that used in prior works, there is one critical difference here. Unlike in the case of the free unitary/free orthogonal quantum groups which rely on the existence of  ``large'' smoooth Lie subgroups (namely $U_N$ and $O_N$, respectively), the quantum reflection groups only admit finite classical subgroups.  This difference turns out to be a fundamental obstruction to a straightforward extension of the inductive arguments of \cite{chi-rfd, bcv}.  To bypass this issue, we make essential use of a recent  remarkable result of Banica \cite[Theorem 7.10]{ban-uni} which establishes that there is no intermediate quantum subgroup for the inclusion $S_5 < S_5^+$.  The maximality of the inclusion $S_N < S_N^+$ is widely conjectured to be true for all $N$, and the case $N=5$ solved by Banica in \cite{ban-uni} represents a major advancement on this conjecture.  It is also interesting to note that Banica's proof of the maximality of the inclusion $S_5 < S_5^+$ is based on a reduction of this problem to the seemingly different problem of classifying II$_1$-subfactors at index $5$.   This latter problem, however, has recently been solved \cite{JoMoSn14, IzMoPePeSn15}.  The authors find this connection to the classification of subfactors highly intriguing.

The other major set of results in this paper pertains to the quantum permutation groups $S_N^+$.  In this case it turns out that we can say quite a lot more at the level of finite-dimensional representations.

While our residual finiteness results ensure the existence of enough finite-dimensional $\ast$-representations to separate points in the group algebras $\cA (S_N^+) = \bC\widehat{S^+_N}$, it is often desirable to have a single representation $\pi:\cA(S_N^+) \to B$, where $B$ is some ``nice'' C$^\ast$-algebra (e.g. finite-dimensional, or of the form $M_N(C(X))$) which encodes enough information about $S_N^+$ so as to generate an {\it asymptotically faithful} sequence of finite-dimensional representations $(\pi_k)_{k \in \bN}$ of $\cA(S_N^+)$.   The relevant concept we are after here is that of an {\it inner faithful} represention $\pi:\cA(S_N^+) \to B$.  We defer the precise definition of inner faithfulness to Section \ref{se.mod} but note here that a representation $\pi:\cA(S_N^+) \to B$ is inner faithful if and only if the sequence of representations 
\[
\pi_k:\cA(S_N^+) \to B^{\otimes k}; \qquad \pi_k = \pi^{\otimes k} \circ \Delta^{(k)};
\]
is  asymptotically faithful in the sense that $\bigcap_{k} \ker \pi_k = \{0\}$ \cite{SkSo16}.  In the above, $\Delta: \cA(S_N^+) \to \cA(S_N^+) \otimes \cA(S_N^+)$ denotes the coproduct and $\Delta^{(k)} = (\id \otimes \Delta^{(k-1)}) \circ \Delta$ for all $k \ge 2$.  In particular, this means that $\cA(S_N^+)$ faithfully embeds into a C$^\ast$-ultraproduct of the sequence of algebras $(B^{\otimes k})_{k \in \bN}$.  One particular ``minimal '' representation of $\cA (S_N^+)$ that has been conjectured to be inner faithful is Banica's {\it universal flat representation} \cite{bn-flat,bf-model}.  This particular representation takes the form $\pi:\cA(S_N^+) \to M_N(C(X_N))$, where $X_N \subset M_N(M_N(\bC))$ is the compact space of all $N \times N$ bistochastic matrices $P= (P_{ij})_{i,j}$ whose entries are rank-one projections in $P_{ij} \in M_N(\bC)$.  In this paper, we use modifications of our topological generation results to verify the conjectured inner faithfulness of the representation $\pi$ for almost all values of $N$ (cf. Corollary \ref{cor.div5}). 

\begin{theoremN}
For all $N \le 5$ and $N \ge 10$, the universal flat matric model $\pi:\cA(S_N^+) \to M_N(C(X_N))$ is inner faithful.
\end{theoremN} 

Piggybacking on the proof of this result, we are able to moreover show that one can reduce the base space $X_N$ to only contain at most $3$ points and still achieve an inner faithful finite-dimensional representation. In other words, we have (cf. Theorem \ref{th.inner-unitary})

\begin{theoremN}
For all $N \le 5$ and $N \ge 10$, the quantum permutation group algebras $\cA(S_N^+)$ are {\it inner unitary}:  they admit inner faithful $\ast$-homomorphisms $\pi:\cA(S_N^+) \to B$, with $B$ a finite dimensional C$^\ast$-algebra.
\end{theoremN}

It is an artifact of our proof that we are unable to settle the cases $N \in [6,9]$ in the above theorems.  Our arguments rely heavily on our topological generation results for $S_N^+$ together with the crucial result of Banica \cite{ban-uni} stating that the inclusion $S_5 < S_5^+$ is maximal.

The remainder of the paper is organized as follows.

\Cref{se.prel} gathers a number of prerequisites to be used later. In \Cref{se.rf} we prove some of the main results, \Cref{th.rf,th.hns}, to the effect that the discrete Pontryagin duals of the quantum reflection groups are residually finite. This then also implies that they are hyperlinear and have the Kirchberg factorization property.  The proofs rely in large part on an inductive argument, turning on the fact that quantum permutation groups $S^+_N$ are topologically generated by their quantum subgroups $S_N$ and $S^+_{N-1}$ (\Cref{th.top-gen}). We also prove a similar result for quantum reflection groups in \Cref{th.refl-tg}; though strictly speaking not needed for residual finiteness of the quantum reflection groups, it might nevertheless be of some independent interest. 

In the final section \ref{se.mod}, we study inner faithful representations of Hopf $\ast$-algebras and prove that the universal flat representations are inner faithful for $N \le 5$ and $N\geq 10$.  In subsection \Cref{subse.inner-unitary} we prove \Cref{th.inner-unitary}, confirming that for sufficiently large $N$ the duals $\cA (S_N^+)$ admit {\it finite-dimensional} inner faithful representations.

\subsection*{Acknowledgements}

M. Brannan and A. Chirvasitu are partially supported by the US National Science Foundation with grants DMS-1700267 and DMS-1801011 respectively.

\section{Preliminaries}\label{se.prel}

\subsection{Generalities} 

Let us start by recalling some facts about compact quantum groups from \cite{Wor98}:

\begin{definition}\label{def.cqg}
  A {\it compact quantum group} is a unital C$^*$-algebra $C(\bG)$ equipped with a unital $\ast$-morphism
  \begin{equation*}
    \Delta: C(\bG)\to C(\bG) \otimes C(\bG)
  \end{equation*}
  (minimal tensor product) such that
  \begin{equation*}
    \mathrm{span}\{(a\otimes 1)\Delta(b)\ |\ a,b\in C(\bG)\} \text{ and }\mathrm{span}\{(1\otimes a)\Delta(b)\ |\ a,b\in C(\bG)\}
  \end{equation*}
are dense in $C(\bG) \otimes C(\bG)$.
\end{definition}

As is customary, we regard $C(\bG)$ as the algebra of continuous functions on the fictitious ``compact quantum space'' $\bG$. For this reason, the category of compact quantum groups is {\it dual} to that of C$^*$-algebras as in \Cref{def.cqg}. 

For a compact quantum group $\bG$ we denote the unique dense Hopf $*$-subalgebra of $C(\bG)$ by $\cA(\bG)$; it can be regarded alternatively as the complex group algebra $\bC\widehat{\bG}$ of the discrete quantum group $\widehat{\bG}$ whose Pontryagin dual is $\bG$, for which reason we might occasionally revert to that alternative notation for it. It is often also rendered as $\mathrm{Pol}(\bG)$ or $\mathcal O(\bG)$ in the literature. 

The C*-algebra $C(G)$ is equipped with a unique state $h:C(\bG)\to \bC$ (its {\it Haar state}) that is left and right-invariant in the sense that the diagram
\begin{equation*}
  \begin{tikzpicture}[baseline=(current  bounding  box.center),anchor=base,cross line/.style={preaction={draw=white,-,line width=6pt}}]
    \path (0,0) node (1) {$C(\bG)$} +(3,.5) node (2) {$C(\bG) \otimes C(\bG)$} +(6,0) node (3) {$C(\bG)$} +(3,-.5) node (4) {$\bC$}; 
    \draw[->] (1) to[bend left=6] node[pos=.5,auto]{$\scriptstyle \Delta$} (2)  ;
    \draw[->] (2) to[bend left=6] node[pos=.5,auto] {$\scriptstyle h\otimes\mathrm{id}$} (3);
    \draw[->] (1) to[bend right=6] node[pos=.5,auto,swap] {$\scriptstyle h$} (4);
    \draw[->] (4) to[bend right=6] node[pos=.5,auto,swap] {} (3);    
  \end{tikzpicture}
\end{equation*}
and the analogous mirror diagram (obtained by substituting $\mathrm{id}\otimes h$ for the upper right hand arrow) both commute. We will also occasionally refer to the {\it quantum group von Neumann algebra} $L^{\infty}(\bG)$, which is by definition the von Neumann algebra generated by the GNS representation of $C(\bG)$ associated to $h$. Often in  the literature the von Neumann algebra $L^\infty(\bG)$ is written as $\mathcal L(\widehat{\bG})$ in view of it being a generalization of the von Neumann algebra generated by the left-regular representation of a discrete group.

\begin{definition}\label{def.rep}
A {\it representation} of the compact quantum group $\bG$ is a finite-dimensional comodule over the Hopf $*$-algebra $\cA(\bG)$. We write $\mathrm{Rep}(\bG)$ for the category of representations of $\bG$. 
\end{definition}

We refer to \cite{rad-bk} for the relatively small amount of background needed here on comodules over coalgebras or Hopf algebras. 

\begin{definition}\label{def:subgroup}
Given two compact quantum groups $\bH$ and $\bG$, we say that $\bH$ is a {\it (closed)  quantum subgroup} of $\bG$ if there exists a surjective Hopf $\ast$-algebra morphism $\pi: \cA(\bG) \to \cA(\bH)$.  In this case, we write $\bH < \bG$.
\end{definition}

Note that if $\bH < \bG$ as above and $V$ is a $\bG$-comodule, then $V$ automatically becomes a $\bH$-comodule in a natural way. Indeed if $\alpha:V \to V \otimes \cA(\bG)$ is the associated corepresentation defining the comodule $V$ and $\pi: \cA(\bG) \to \cA(\bH)$ is the surjective Hopf $\ast$-morphism from above, then $V$ becomes an $\bH$-comodule via the corepresentation $(\id \otimes \pi)\alpha: V \to V \otimes \cA(\bH)$.  This ``restriction'' of representations of $\bG$ to representations of $\bH$, induces, at the level of Hom-spaces, natural inclusions  
\begin{equation*}
    \mathrm{hom}_{\bG}(V,W)\hookrightarrow \mathrm{hom}_{\bH}(V,W),
  \end{equation*}
for any pair of $\bG$ comodules $V,W$.

\subsection{Quantum reflection groups}\label{subse.refl} 

Quantum reflection groups were introduced in \cite[Definition 1.3]{bv-reflection} based on earlier work done in \cite{bbcc}. Following the former reference we denote them by $H^{s+}_N$.

\begin{definition}\label{def.refl}
  Let $N\ge 2$ and $1\le s\le \infty$.
  
  The underlying Hopf $*$-algebra $\cA=\cA(H^{s+}_N)$ of the {\it quantum reflection group} $H^{s+}_N$ is generated as a $*$-algebra by $N^2$ normal elements $u_{ij}$, $1\le i,j\le N$ such that
  \begin{itemize}
  \item the matrices $u=(u_{ij})_{i,j}$ and $u^t = (u_{ji})_{i,j} $ are both unitary in $M_N(\cA)$; 
   \item $p_{ij}=u_{ij}u^*_{ij}$ is a self-adjoint idempotent for each $1 \le i,j \le N$;
  \item $u^s_{ij}=p_{ij}$ for each $1 \le i,j \le N$,   
  \end{itemize}
  with the last relation absent when $s=\infty$.  

The coproduct $\Delta: \cA \to \cA \otimes \cA$ is determined by \[\Delta(u_{ij}) = \sum_{k=1}^N u_{ik} \otimes u_{kj} \qquad (1 \le i,j \le N).\]
\end{definition}

The quantum groups $H^{s+}_N$ are meant to be quantum analogues of their classical versions $H^s_N$ consisting of monomial $N\times N$ matrices whose non-zero entries are $s^{th}$ roots of unity. In particular, $s=1$ recovers the symmetric group $S_N$; this is also the case in the quantum setting, recovering the quantum groups introduced in \cite{Wan98}: 

\begin{definition}\label{def.sn}
Let $N\ge 2$. The {\it quantum symmetric group} $S^+_N$ is the quantum reflection group $H^{1+}_N$ from \Cref{def.refl}. Its Hopf $\ast$-algebra $\cA$ is freely generated as a $*$-algebra by an $N \times N$ {\it magic unitary matrix} $u=(u_{ij})_{ij}\in M_N(\cA)$, in the sense that all of the entries of $u$ are projections and the entries from each row and column add up to $1\in\cA$.   
\end{definition}

Now let $N\ge 2$ and $1\le s,t<\infty$ be positive integers, as in \Cref{def.refl}. The generators $u_{ij}$ of $\cA(H^{s+}_N)$ clearly satisfy the defining relations of $\cA(H^{st+}_N)$ as well, so we get a surjective Hopf $*$-algebra morphism
\begin{equation*}
  \cA(H^{st+}_N)\ni u_{ij}\mapsto u_{ij}\in \cA(H^{s+}_N). 
\end{equation*}
In other words, we have natural embeddings $H^{s+}_{N} < H^{st+}_N$.  These embeddings will feature again below.

Also for future use, we briefly recall some background on the representation theory of the quantum groups $H^{s+}_N$ from \cite{bv-reflection}. Let $F_s$ be the free monoid on the symbols in $\bZ_s$ (understood as $\bZ$ when $s=\infty$), equipped with the following operations:

\begin{itemize}
\item The involution $x\mapsto \overline{x}$ defined by
  \begin{equation*}
    a_1\cdots a_k \mapsto (-a_k)\cdots (-a_1),\ a_i\in \bZ_s;
  \end{equation*}
\item The {\it fusion} binary operation `$\cdot$' defined by
  \begin{equation*}
    (a_1\cdots a_k)\cdot(b_1\cdots b_\ell) = a_1\cdots a_{k-1}(a_k+b_1)b_2\cdots b_\ell. 
  \end{equation*}  
\end{itemize}

Then, according to \cite[Theorem 7.3]{bv-reflection}, the Grothendieck ring $R_s$ of the category of finite-dimensional representations of $H^{s+}_N$ has a basis $\{x_f\}$ over $\bZ$ indexed by $f\in F_s$, and the multiplication resulting from the tensor product
\begin{equation*}
  x_f x_g = \sum_{f=vz,g=\overline{z}w}\left(x_{vw}+x_{v\cdot w}\right). 
\end{equation*}
We remark that \cite[Theorem 7.3]{bv-reflection} is not clearly stated in this manner for $s=\infty$, but the proof goes through essentially unchanged for both finite and infinite $s$. A more general argument can be found in \cite{lemeux2013fusion}. 

Restriction via the inclusion $H^{s+}_N< H^{st+}_N$ induces a ring morphism $R_{st}\to R_s$ sending the generator $x_1$, $1\in \bZ_{st}$ to $x_1$, $1\in \bZ_s$. In particular, we have

\begin{lemma}\label{le.limr}
  The natural map
  \begin{equation*}
    R_{\infty}\to \varprojlim_s R_s
  \end{equation*}
  is an embedding, where the directed inverse limit is taken over the positive integers ordered by division: $s\le st$. \qedhere
\end{lemma}

\subsection{Finiteness properties and topological generation}

The following is a combination of \cite[Definition 2]{bcv} and \cite[Definition 1.12]{bbcw}. 

\begin{definition}\label{def.fin}
A discrete quantum group $\widehat{\bG}$ is {\it finitely generated} if $\bC \widehat{\bG}$ is finitely generated as an algebra. If $\widehat{\bG}$ is finitely generated then we say it is {\it residually finite} if $\bC\widehat{\bG}$ embeds as a $*$-algebra into a (possibly infinite) direct product of matrix algebras. 

Moreover, $\widehat{\bG}$ is said to have the {\it Connes embedding property} (or is {\it Connes-embeddable} or {\it hyperlinear}) if it is of Kac type and the von Neumann algebra $(L^{\infty}(\bG)$ admits a Haar state-preserving embedding into the   ultrapower $R^{\omega}$ of the hyperfinite II$_1$-factor. 
\end{definition}

\begin{remark}\label{rem.RFD}
Note that residual finiteness implies the {\it Kirchberg factorization property} of \cite[Definition 2.10]{bw} by \cite[Theorem 2.1]{bbcw}. In turn, Kirchberg factorization implies hyperlinearity (e.g. \cite[Remark 2.6]{bbcw}); in conclusion, residual finiteness is stronger than Connes embeddability for a discrete quantum group. 
\end{remark}

We also recall from \cite[Definition 4]{bcv} the notion of topological generation for compact quantum groups:

\begin{definition}\label{def.top-gen}
  A family of compact quantum subgroups $(\bG_i < \bG)_{i \in I}$ {\it topologically generate} $\bG$ if, for every pair of representations $V$, $W$ of $\bG$, the natural inclusion map
  \begin{equation*}
    \mathrm{hom}_{\bG}(V,W)\hookrightarrow \bigcap_{i \in I}  \mathrm{hom}_{\bG_i}(V,W)
  \end{equation*}
  is an isomorphism.  In this case, we write $\bG = \langle \bG_i\rangle_{i \in I}$
\end{definition}

We will need the following alternative description of topological generation, which is almost immediate given \Cref{def.top-gen}; see also \cite[Proposition 3.5]{bcv}.

\begin{lemma}\label{le.alt-top-gen}
A family of quantum subgroups $(\bG_i < \bG)_{i \in I}$ topologically generates $\bG$ if and only if for every  $\bG$-representation $V$ a map $f:V\to \bC$ that is a morphism over every $\bG_i$ is a morphism over $\bG$ (i.e., $f  \in \cap_{i \in I} \mathrm{hom}_{\bG_i}(V,\bC) \implies f  \in \mathrm{hom}_{\bG}(V,\bC) $). Equivalently, it suffices to check this for all irreducible representations $V$.
  \qedhere
\end{lemma}

In particular, we have the following sufficient criterion for topological generation:

\begin{corollary}\label{cor.tg-suff}
  Let $(\bG_i< \bG)_{i \in I}$ be a family of quantum subgroups of a compact quantum group and assume that for every irreducible non-trivial $\bG$-representation $V$ there is some $i$ such that the restriction of $V$ to $\bG_i$ contains no trivial summands.

  Then, $\bG$ is topologically generated by the $\bG_i$. 
\end{corollary}
\begin{proof}
  This follows from \Cref{le.alt-top-gen}: a morphism $f:V\to \bC$ of representations witnesses an embedding of the trivial representation into $V$, and the hypothesis ensures that a non-trivial irreducible $V\in \mathrm{Rep}(\bG)$ retains the property of having no trivial summands over $\bG_i$ for some $i$, meaning that
  \begin{equation*}
    \bigcap_{i \in I} \mathrm{hom}_{\bG_i}(V,\bC) = \{0\} = \mathrm{hom}_{\bG}(V,\bC). 
  \end{equation*}
\end{proof}

Topological generation appears under a different name in \cite[$\S$2]{chi-rfd}. Rephrasing Corollary 2.16 therein more appropriately for our setting yields

\begin{lemma}\label{le.gen-rf}
  A finitely generated discrete quantum group $\widehat{\bG}$ is residually finite if and only if its dual $\bG$ is topologically generated by a family of subgroups $\bG_i< \bG$ with residually finite $\widehat{\bG_i}$.
  \qedhere
\end{lemma}

Recall the embeddings $H^{s+}_N< H^{st+}_N$ from \Cref{subse.refl}. The remark we will need in the sequel is

\begin{lemma}\label{le.hn-top-gen}
For $N\ge 2$ the quantum group $H^{\infty +}_N$ is topologically generated by its quantum subgroups $H^{s+}_N$ for finite $s$. 
\end{lemma}
\begin{proof}
  According to \Cref{le.limr} the criterion of \Cref{cor.tg-suff} is satisfied: indeed, the former result ensures that every non-trivial irreducible $H^{\infty+}_N$-representation remains irreducible and non-trivial over some $H^{s+}_N$.
\end{proof}

\subsection{Residual finite-dimensionality}\label{subse.rfd}  

We focus here on $*$-algebras satisfying the condition required of $\bC \widehat{\bG}$ in \Cref{def.fin}:

\begin{definition}\label{def.rfd}
A $*$-algebra $\cA$ is {\it residually finite-dimensional} or {\it RFD} if it embeds as a $*$-algebra in a product of matrix algebras.   
\end{definition}

\begin{remark}\label{re.rfd}
The notion is used frequently in the context of C$^*$-algebras, but here we are interested in the purely $*$-algebraic version.   
\end{remark}

We gather a number of general observations on the RFD property for later use. First, since residual finite-dimensionality obviously passes to $*$-subalgebras, we will need to know that certain natural morphisms between free products with amalgamation (or {\it pushouts}, as we will also refer to them) are embeddings. The following result is likely well known, but we include it here for completeness. 

\begin{lemma}\label{le.emb-push}
  Suppose we have the following commutative diagram of complex algebras, all of whose arrows are embeddings.
  \begin{equation*}
    \begin{tikzpicture}[baseline=(current  bounding  box.center),anchor=base,cross line/.style={preaction={draw=white,-,line width=6pt}}]
      \path (0,0) node (1l) {$\cD$} +(2,1) node (2l) {$\cA$} +(2,-1) node (3l) {$\cB$}
      +(4,0) node (1r) {$\cD$} +(6,1) node (2r) {$\cA'$} +(6,-1) node (3r) {$\cB'$}; 
      \draw[->] (1l) to[bend left=6] (2l);
      \draw[->] (1l) to[bend right=6] (3l);
      \draw[->] (1r) to[bend left=6] (2r);
      \draw[->] (1r) to[bend right=6] (3r);
      \draw[->] (1l) to node[pos=.5,auto] {$\scriptstyle \cong$} (1r);
      \draw[->] (2l) to (2r);
      \draw[->] (3l) to (3r);      
  \end{tikzpicture}
\end{equation*}
If $\cD$ is finite-dimensional and semisimple then the canonical map $\cA*_{\cD}\cB\to \cA'*_{\cD}\cB'$ is one-to-one.
\end{lemma}
\begin{proof}
Denote by  $\cD^{\circ}$ the opposite algebra of $\cD$.  Under the hypothesis on $\cD$ the enveloping algebra $\cD\otimes \cD^{\circ}$ is semisimple, and hence the category of $\cD$-bimodules (which are simply $\cD\otimes \cD^{\circ}$-modules) is semisimple. It follows that the inclusions
  \begin{equation*}
    \cD\to \cA\to \cA'
  \end{equation*}
are split as $\cD$-bimodule maps, and hence we have direct sum decompositions
\begin{equation}\label{eq:dab}
  \cA = \cD\oplus \cA_1,\  \cA'=\cD\oplus \cA'_1 = \cD\oplus \cA_1\oplus \cA_2 
\end{equation}
and similarly for the $\cB$ side of the diagram. 

According to \cite[Corollary 8.1]{bg-diamond} the pushout $\cA*_{\cD}\cB$ decomposes as a direct sum of tensor products of the form
\begin{equation}\label{eq:ts}
  T_1\otimes T_2\cdots \otimes T_k,\quad T_i\text{ chosen alternately from } \{\cA_1,\cB_1\}
\end{equation}
in the category of $\cD$-bimodules. The analogous decomposition holds for $\cA'*_{\cD}\cB'$, and due to \Cref{eq:dab} the tensor product \Cref{eq:ts} is a summand in its counterpart
\begin{equation*}
  T'_1\otimes \cdots\otimes T'_k,\quad T'_i\in \{\cA'_1,\cB'_1\}\text{ alternately},\quad T'_i=\cA'_1 \iff T_i=\cA_1. 
\end{equation*}
The conclusion that $\cA*_{\cD}\cB\to \cA'*_{\cD}\cB'$ is an embedding follows.
\end{proof}

There are analogues of this in the C$^*$-algebra literature: see e.g. \cite[Proposition 2.2]{adel} and \cite[Theorem 4.2]{ped-psh}. We will only use \Cref{le.emb-push} in the context of $*$-algebras, with all embeddings being $*$-algebra morphisms. An immediate consequence of \Cref{le.emb-push} is

\begin{corollary}\label{cor.rfd-psh}
  Under the hypotheses of \Cref{le.emb-push}, suppose furthermore that all maps are $*$-algebra morphisms. If $\cA'*_{\cD}\cB'$ is RFD, then so is $\cA*_{\cD}\cB$. 
  \qedhere
\end{corollary}

We now turn to the technical result of this section, which will be needed later on.

\begin{proposition}\label{pr.aa}
  If $\cA$ is an RFD $*$-algebra and $\cD\subset \cA$ is a finite-dimensional commutative C$^*$-subalgebra then $\cA*_{\cD}\cA$ is RFD. 
\end{proposition}

Before going into the proof we treat a particular case.

\begin{lemma}\label{le.aa-fin}
\Cref{pr.aa} holds when $\cA$ is a finite-dimensional C$^*$-algebra. 
\end{lemma}

This in turn requires some preparation. More precisely, we first build a faithful Hilbert space representation of the pushout $\cA*_{\cD}\cA$. Throughout the present discussion $\cD\subset \cA$ are as in the statement of \Cref{le.aa-fin}. Consider the canonical conditional expectation $E:\cA\to \cD$ that preserves an arbitrary but fixed faithful tracial state $\tau$ on $\cA$. This means that
  \begin{equation*}
    \begin{tikzpicture}[baseline=(current  bounding  box.center),anchor=base,cross line/.style={preaction={draw=white,-,line width=6pt}}]
      \path (0,0) node (1) {$\cA$} +(2,.5) node (2) {$\cD$} +(4,0) node (3) {$\bC$}; 
      \draw[->] (1) to[bend left=6] node[pos=.5,auto] {$\scriptstyle E$} (2);
      \draw[->] (2) to[bend left=6] node[pos=.5,auto] {$\scriptstyle \tau|_{\cD}$} (3);
      \draw[->] (1) to[bend right=6] node[pos=.5,auto,swap] {$\scriptstyle \tau$} (3);      
    \end{tikzpicture}
  \end{equation*}
Moreover, the map $E$ is contractive, completely positive, and splits the inclusion $\cD\to \cA$ in the category of $\cD$-bimodules (see e.g. \cite[$\S$IX.4]{Tak2} for background on expectations on operator subalgebras). 

We denote by $\cA_{\ell}$ and $\cA_r$ the copies of $\ker(E)$ in the left and respectively right hand side free factor of $\cA*_{\cD}\cA$. Then, as a consequence of \cite[Corollary 8.1]{bg-diamond}, we have
\begin{equation*}
  \cA*_{\cD}\cA = \bigoplus_{{\bf i}} T_{{\bf i}}
\end{equation*}
where ${\bf i}$ ranges over the words on ${\ell,r}$ with no consecutive repeating letters and e.g.
\begin{equation*}
  T_{\ell r\ell \cdots} = \cA_{\ell}\otimes \cA_r\otimes \cA_{\ell}\otimes\cdots
\end{equation*}
with tensor products over $\cD$. The empty word is allowed, with $T_{\emptyset}=\cD$.

Each $T_{\bf i}$ is then naturally a Hilbert $\cD$-module (e.g. \cite[Chapter 15]{wo}). Composing the $\cD$-valued inner product on $T_{\bf i}$ further with the inner product
\begin{equation*}
 \braket{x|y} = \tau(x^*y),\ x,y\in \cD 
\end{equation*}
makes each $T_{\bf i}$ into a Hilbert space. Left multiplication makes the algebraic direct sum $\bigoplus_{\bf i}T_{\bf i}$ a faithful module over $\cA*_{\cD}\cA$ (it is simply the left regular representation of the algebra in question). The Hilbert space completion of $\bigoplus_{\bf i} T_{\bf i}$ is thus a faithful Hilbert space representation of $\cA*_{\cD} \cA$.  

\pf{le.aa-fin}
\begin{le.aa-fin}
  The fact that the full C$^*$-pushout $\overline{\cA*_{\cD}\cA}$ is residually finite-dimensional follows from \cite[Theorem 4.2]{adel}. It thus remains to show that the algebraic pushout $\cA*_{\cD}\cA$ embeds in its C$^*$-envelope.

  In other words, we want to argue that $\cA*_{\cD}\cA$ admits a faithful representation on some Hilbert space. This, however, is precisely what we constructed above in the discussion preceding the proof: the Hilbert space direct sum of $T_{\bf i}$ for ${\bf i}$ ranging over words on $\{\ell,r\}$ with no repeating consecutive letters is such a representation. 
\end{le.aa-fin}

\begin{remark}
  The construction of the Hilbert modules $T_{\bf i}$ sketched above features prominently in free probability; see e.g. \cite[$\S$5]{voi-sym} and \cite[$\S$3.8]{vd-free}. 
\end{remark}

\pf{pr.aa}
\begin{pr.aa}
  The RFD condition implies that there is an embedding $\cA\to \prod_{i\in I} M_{n_i}$ for an index set $I$ (or arbitrary cardinality). By \Cref{cor.rfd-psh}, it thus suffices to assume that $\cA$ {\it is} such a product of matrix algebras to begin with. We therefore make this assumption throughout the rest of the proof.

  We first appeal once more to \cite[Corollary 8.1]{bg-diamond} to conclude that since the embedding $\cD\to \cA$ splits as
  \begin{equation*}
    \cA= \cD\oplus \cA_1
  \end{equation*}
  in the category of $\cD$-bimodules (because $\cD\otimes \cD^{\circ}$ is semisimple), we have a decomposition
  \begin{equation}\label{eq:t1k}
    \cA*_{\cD}\cA\cong \bigoplus T_1\otimes \cdots \otimes T_k 
  \end{equation}
  where $T_i$ are chosen from among the two copies of $\cA_1$. An arbitrary element in $\cA*_{\cD}\cA$ can thus be expressed as a sum of elements of finitely many tensor products as in \Cref{eq:t1k}. It follows that if $x$ is non-zero then it maps to a non-zero element of a pushout $\cB*_{\cD}\cB$ through some projection
  \begin{equation*}
    \cA=\prod_{i\in I}M_{n_i}\to \prod_{i\in F}M_{n_i}=\cB
  \end{equation*}
  for a finite subset $F\subseteq I$. In conclusion, it will be enough to further assume that $\cA$ is a finite product of matrix algebras, i.e. a finite-dimensional C$^*$-algebra. This is \Cref{le.aa-fin}, hence the conclusion. 
\end{pr.aa}

\begin{remark}\label{re.li-shen}
  Once more, there are versions of \Cref{pr.aa} applicable to C$^*$-algebras; \cite[Corollary 2]{ls-rfd} is one example.
\end{remark}

We conclude this section with one more proposition which will have direct application in the next section to the residual finiteness of the duals of quantum reflection groups.

\begin{proposition}\label{pr.push-rf}
  Let $\cB$ be an RFD $*$-algebra, $\cD\subset \cB$ a finite-dimensional commutative C$^*$-subalgebra, and $\cC$ a finite-dimensional C$^*$-algebra. Then,
  \begin{equation*}
    \cC*\cB/[\cC,\cD]
  \end{equation*}
  is RFD. 
\end{proposition}
\begin{proof}
  The algebra in question is isomorphic to the pushout
  \begin{equation*}
    (\cC\otimes \cD)*_{\cD} \cB. 
  \end{equation*}
  Since we have rightward embeddings
  \begin{equation*}
    \begin{tikzpicture}[baseline=(current  bounding  box.center),anchor=base,cross line/.style={preaction={draw=white,-,line width=6pt}}]
      \path (0,0) node (1l) {$\cD$} +(2,1) node (2l) {$\cC\otimes\cD$} +(2,-1) node (3l) {$\cB$}
      +(4,0) node (1r) {$\cD$} +(6,1) node (2r) {$\cC\otimes \cB$} +(6,-1) node (3r) {$\cC\otimes \cB$}; 
      \draw[->] (1l) to[bend left=6] (2l);
      \draw[->] (1l) to[bend right=6] (3l);
      \draw[->] (1r) to[bend left=6] (2r);
      \draw[->] (1r) to[bend right=6] (3r);
      \draw[->] (1l) to node[pos=.5,auto] {$\scriptstyle \cong$} (1r);
      \draw[->] (2l) to (2r);
      \draw[->] (3l) to (3r);      
  \end{tikzpicture}
  \end{equation*}

\Cref{cor.rfd-psh} shows that it will be enough to prove residual finite-dimensionality for
  \begin{equation*}
   (\cC\otimes \cB)*_{\cD}(\cC\otimes \cB) 
  \end{equation*}
In turn, this follows from \Cref{pr.aa} and the simple observation that tensor products of RFD $*$-algebras are RFD.
\end{proof}

\section{Topological generation and residual finiteness}\label{se.rf}

Our first goal in this section is to treat the case of quantum permutation groups.  Before getting to the main results, we first recall some basic facts about the description of the invariant theory for $S_N^+$ in terms of non-crossing partitions \cite{BaCo07}.

 \subsection{Non-crossing partition maps} \label{section:ncp}

\begin{definition}
Fix $k \in \bN$ and consider the ordered set $[k] = \{1, \ldots, k\}$.  A {\it partition} of $[k]$ is a decomposition $p$ of $[k]$ into a disjoint union of non-empty subsets, called the {\it blocks} of $p$. A partition $p$ of $[k]$ has a {\it crossing} if there exist $a < b < c <d \in [k]$ such that $\{a,c\}$ and $\{b,d\}$ belong to different blocks.  A partition $p$ is called {\it non-crossing} if it has no crossings.  The collection of all non-crossing partions of $[k]$ is denoted by $NC(k)$.  
\end{definition} 

Given a function $i:[k] \to [N]$ (i.e. a multi-index $i = (i(1), \ldots, i(k)) \in [N]^k$), we let $\ker i$ be the partition of $[k]$ given by declaring that $r,s$ belong to the same block of $\ker i$ if and only if $i(r) = i(s)$.  Given $i$ as above and $p \in NC(k)$, we define \[\delta_p(i) =
 \Bigg\{\begin{matrix} 1& \ker i \ge p \\0& \text{otherwise}
\end{matrix},\]
where $\ge$ denotes the refinement partial order on the lattice partitions of $[k]$.  

Now let $V$ be an $N$-dimensional Hilbert space with distinguished orthonormal basis $e_j$, $1\le j\le N$.
Given $k,l \in \bN_0$ and $p \in NC(k + l)$, we form the linear map 
\[
T_p^{k,l,N}: V^{\otimes k} \to V^{\otimes l}; \qquad T_p^{k,l,N}(e_{i(1)} \otimes \ldots \otimes e_{i(k)}) = \sum_{j:[l] \to [N]}
\delta_p(ij) e_{j(1)}\otimes \ldots \otimes e_{j(l)},
\] 
where $ij:[k+l] \to [N]$ is the concatenation of $i$ and $j$.  We call such linear maps {\it non-crossing partition maps}.  The fundamental result that will be of use to us here is the following description of $S_N^+$ invariants in terms of these non-crossing partition maps.

\begin{theorem}[\cite{BaCo07}]
For each $k,l \in \bN_0$, $N \in \bN$ we have linear isomorphisms
\[
\textrm{hom}_{S_N^+}(V^{\otimes k}, V^{\otimes l}) = \text{span}\{T_p^{k,l,N}: p \in NC(k,l)\}.
\]
Moreover, if $N \ge 4$, then $\{T_p^{k,l,N}: p \in NC(k,l)\}$ forms a linear basis for $\text{hom}_{S_N^+}(V^{\otimes k}, V^{\otimes l}) $.
\end{theorem}

\subsection{Quantum permutation groups}\label{subse.sn} 

Let $V = \text{span}(e_i)_{i=1}^N$ be as above, regarded as the fundamental representation of $S_N^+$.  We regard $S_{N-1}^+$ as also acting on $V$ via its fundamental representation on the $(N-1)$-dimensional subspace $V_N$ spanned by $e_i$, $1\le i\le N-1$, and via the trivial representation on the orthogonal  complement $\bC e_N$.  In this way, we regard $S_{N-1}^+  < S_N^+$.  For future reference, we denote by $V_i$ the subspace of $V$ spanned by all $e_j$, $j\ne i$, and by $W_i$ the span of $e_i$ alone.

Our first main result in this section reads as follows. 

\begin{theorem}\label{th.top-gen}
  For $N\ge 5$ the quantum permutation group $S^+_N$ is topologically generated by its subgroups $S_N$ and $S_{N-1}^+$. 
\end{theorem}

Before embarking on the proof we record the following immediate strengthening of the statement:

\begin{corollary}\label{cor.top-gen-bis}
  Let $N\ge 5$, $4\le M\le N$ and $S^+_M<S^+_N$ the embedding of the quantum subgroup fixing $N-M$ of the standard basis vectors of the defining representation of $S^+_N$ on $\bC^N$. Then, $S^+_N$ is topologically generated by $S_N$ and $S^+_M$. 
\end{corollary}
\begin{proof}
This is a repeated application of \Cref{th.top-gen}: $S_M < S_N$ and $S^+_M$ topologically generate $S^+_{M+1}$ and the result follows by induction.  
\end{proof}

The cases $N=5$ and $N\ge 6$ will be treated differently. The latter requires \Cref{pr.coinv} below, where $V$ is the $N$-dimensional Hilbert space with distinguished orthonormal basis $e_j$, $1\le j\le N$ carrying the defining representation of $S^+_N$.

\begin{proposition}\label{pr.coinv}
  Suppose $N\ge 6$. For any $k\ge 0$, any linear map $f:V^{\otimes k}\to \bC$ that is invariant under both $S_{N-1}^+$ and the classical permutation group $S_N$ is invariant under all of $S_N^+$. 
\end{proposition}
\begin{proof}
In other words, we have to show a functional $f:V^{\otimes k}\to \bC$ that is  both an $S_{N-1}^+$-coinvariant and an $S_N$-coinvariant must also be an $S_N^+$-coinvariant.  

Note that the $S_N$-invariance ensures that $f$ respects the action of {\it each one} among the $N-1$ choices of quantum subgroup $S_{N-1}^+< S_N^+$ obtained by acting on the subspaces $V_i \subset V$. 
Moreover, the invariance under the original copy of $S_{N-1}^+< S_N^+$ ensures that when restricted to $V_N^{\otimes k}$, $f$ acts as a linear combination of the non-crossing partition maps $\{T^{k,0,N-1}_p\}_{p \in NC(k)}$. Replacing the maps $\{T^{k,0,N-1}_p\}_{p \in NC(k)}$  in this linear combination with $\{T^{k,0,N}_p\}_{p \in NC(k)}$ (which belong to $\textrm{hom}_{S_N^+}(V^{\otimes k}, \bC)$) and subtracting this new linear combination from $f$, we may as well assume that $f|_{V_N^{\otimes k}}$ vanishes and then try to prove that $f$ itself is zero. 

Once more, the $S_N$-invariance ensures that the restriction of $f$ to $V_i^{\otimes k}\subset V^{\otimes k}$ vanishes for every $1\le i\le N$. What we have to show, however, is that it also vanishes on summands of $V^{\otimes k}$ obtained by tensoring some copies of $V_N$ with some copies of $W_N$. To simplify notation and fix ideas, we will show that the restriction of $f$ to, say, 
\begin{equation*}
U=V_N^{\otimes(k-l)}\otimes W_N^{\otimes l} \subset V  
\end{equation*}
vanishes. The general case is perfectly analogous, with only notational difficulties making the presentation more cumbersome. 

As the action of the original copy of $S_{N-1}^+$ that we considered is trivial on $W_N$, $U$ can be identified with the $(k-l)^{\text{th}}$ tensor power of the fundamental representation of $S_{N-1}$, and hence any $S_{N-1}^+$-coinvariant $U\to\bC$ will be some linear combination of non-crossing partitions maps associated to $NC(k-l)$. 

Now let $V_{1,N}\subset V_N$ be the span of $e_j$, $j\ne 1,N$. Because $N\ge 6$, $V_{1,N}$ is at least $4$-dimensional and hence the linear forms associated to non-crossing partitions are linearly independent on it. This means that if a linear combination of non-crossing partition functionals on $V_{1,N}^{\otimes (k-l)}\to \bC$ vanishes, then the linear combination itself must be trivial. But note now that $f$ restricted to $V_{1,N}^{\otimes(k-l)}\otimes W_N^{\otimes l}$ vanishes, because the space in question is a subspace of $V_1^{\otimes k}$. By the paragraph above, $f$ must therefore vanish on all of $U$.  
\end{proof}

\pf{th.top-gen}
\begin{th.top-gen}
  As mentioned above, we treat the cases $N=5$ and $N>5$ separately.

  {\bf (Case 1: $N\ge 6$)} Recall e.g. from \cite[Theorem 4.1]{ban-sym} that every finite-dimensional $S^+_N$-representation appears as a summand of $V^{\otimes k}$ where $V$ is the $N$-dimensional defining representation and $k$ is some positive integer. By \Cref{le.alt-top-gen} the conclusion is now a paraphrase of \Cref{pr.coinv}. 

  {\bf (Case 2:  $N=5$)} According to \cite[Theorem 7.10]{ban-uni} the inclusion $S_5< S^+_5$ admits no intermediate quantum groups. Since $S^+_4< S^+_5$ is not a quantum subgroup of $S_5$, we indeed have $S^+_5=\langle S_5,S^+_4\rangle$.
\end{th.top-gen}

As a consequence of the above we have

\begin{theorem}\label{th.rf}
  The discrete duals $\widehat{S^+_N}$ of the free quantum permutation groups are residually finite. 
\end{theorem}
\begin{proof}
  By \Cref{le.gen-rf,th.top-gen} we can proceed inductively once we know that
  \begin{itemize}
  \item $\widehat{S_N}$ is residually finite (the group algebra is finite-dimensional);
  \item $\widehat{S^+_4}$ is residually finite.    
  \end{itemize}
  For the latter, recall from \cite[Definition 2.1 and Theorem 4.1]{bc-pauli} that $\cA(S^+_4)$ embeds into the C$^*$-algebra $C(SU_2, M_4)$, and hence has enough $4$-dimensional representations.
\end{proof}

Since, as observed in Remark \ref{rem.RFD}, hyperlinearity and the Kirchberg factorization property are weaker than residual finiteness, we also have

\begin{corollary}\label{cor.hyp}
  The discrete duals $\widehat{S^+_N}$ have the Kirchberg factorization property and are hyperlinear.
  \qedhere
\end{corollary}

It will be of some interest to have alternative topological generation results which we now state and prove.
Let $4\le M\le N$ be a pair of positive integers, and write $N=M+T$. We then have Hopf $\ast$-algebra surjections
\begin{equation*}
  \cA(S^+_N)\to \cA(S^+_M)*\cA(S^+_T)
\end{equation*}
that annihilate $u_{ij}$ for $i,j$ in distinct parts of any partition of $[N]$ into two parts of sizes $M$ and $T$. We will be somewhat vague on which partitions to use; sometimes we need to refer to arbitrary ones, but when we do not, the reader can simply assume the partition is
\begin{equation}\label{eq:part}
  [N] = [M] \sqcup \{M+1,\cdots,M+T\}, 
\end{equation}
corresponding to the upper left-hand corner embedding
\begin{equation*}
  S_M^+ < S_N^+
\end{equation*}
corresponding to action of $S_M^+$ on $V = \text{span}\{e_1, \ldots, e_N\}$ which fixes $e_{M+1}, \ldots, e_{M+T}$.

We now come to a critical notion.

\begin{definition}\label{def.large}
  An compact quantum group embedding $\bG< S^+_N$ is {\it $(M,N)$-large} if it factors the upper path in the diagram
  \begin{equation}\label{eq:vee}
    \begin{tikzpicture}[baseline=(current  bounding  box.center),anchor=base,cross line/.style={preaction={draw=white,-,line width=6pt}}]
    \path (0,0) node (1) {$\cA(S^+_N)$} +(4,-.5) node (2) {$\cA(S^+_M)*\cA(S^+_T)$} +(7,-.5) node (3) {$\cA(\bG)$} +(0,-1) node (1bis) {$\cA(S^+_M)$}; 
    \draw[->] (1) to [bend right=6] (2);
    \draw[->] (2) -- (3);    
    \draw[->] (1bis) to[bend left=6] (2);
  \end{tikzpicture}
  \end{equation}
so as to make its lower path one-to-one. 
\end{definition}

\begin{example}
  The obvious examples of $(M,N)$-large embeddings are those corresponding to the standard surjections $\cA(S^+_N)\to \cA(S^+_M)$. Slightly less obvious examples can be obtained as follows: Suppose $N=KM$ for some positive integer $K$.   The {\it diagonal embedding} $S^+_M\le S^+_N$ is obtained at the level of Hopf algebras as the surjection
  \begin{equation*}
    \cA(S^+_N)\to \cA(S^+_M)^{*K}\to \cA(S^+_M)
  \end{equation*}
where the left hand arrow annihilates the generators $u_{ij}$ that are off the diagonal consisting of $M\times M$ blocks and the right hand arrow is the identity on each free factor. Diagonal embeddings are $(M,N)$-large in the sense of \Cref{def.large}. 
\end{example}

The embedding of Hopf $\ast$-algebras $\cA(S^+_M)\to \cA(\bG)$ forming the lower half of \Cref{eq:vee} corresponds to a morphism of quantum groups $\bG\to S^+_M$, which in turn gives rise to restriction and induction functors
\begin{equation}\label{eq:frob}
  \begin{tikzpicture}[baseline=(current  bounding  box.center),anchor=base,cross line/.style={preaction={draw=white,-,line width=6pt}}]
    \path (0,0) node (1) {$\mathrm{Rep}(S^+_M)$} +(4,0) node (2) {$\mathrm{Rep}(\bG)$}; 
    \draw[->] (1) to [bend left=6] node[pos=.5,auto] {$\scriptstyle \mathrm{res}$} (2);
    \draw[->] (2) to [bend left=6] node[pos=.5,auto] {$\scriptstyle \mathrm{ind}$} (1);
  \end{tikzpicture}
\end{equation}
with restriction being the left adjoint to induction. These remarks will recur below.

\begin{proposition}\label{pr.gen-diag}
Let $5\le M\le N$ and $\bG< S^+_N$ an $(M,N)$-large embedding. Then, $S^+_N$ is topologically generated by $\bG$ and $S_N$.   
\end{proposition}

\begin{proof}
Throughout the proof we will fix a basis $\{e_i\}$ for the $N$-dimensional carrier space $V$ of the defining representation of $S^+_N$, and assume that the upper left hand arrow in \Cref{eq:vee} is the standard one corresponding to the partition \Cref{eq:part}. The proof will be very similar to the that of \Cref{pr.coinv}: we fix a map $f:V^{\otimes k}\to \bC$ that is a morphism both over $S_N<S^+_N$ and $\bG$ and seek to show that $f$ is also an $S^+_N$-morphism.    

Decompose $V=V_1\oplus V_2$ where
\begin{equation*}
  V_1=\mathrm{span}\{e_1,\cdots,e_M\},\quad V_2=\mathrm{span}\{e_{M+1},\cdots,e_{M+T}\}. 
\end{equation*}
The tensor power $V^{\otimes k}$ then decomposes as
\begin{equation*}
  \bigoplus_{{\bf i}} V_{i_1}\otimes\cdots \otimes V_{i_k}
\end{equation*}
with summands ranging over all tuples ${\bf i}=(i_j)$, $i_j\in \{1,2\}$. The summand $V_1^{\otimes k}$ is a comodule over $\cA(S^+_M)\subset \cA(\bG)$ (the embedding being the lower path in \Cref{eq:vee}), and hence $f|_{V_1^{\otimes k}}$ is a linear combination of non-crossing partitions. Subtracting the same linear combination of non-crossing partitions on $V^{\otimes k}$, we can assume $f|_{V_1^{\otimes k}}=0$. The goal now is to show that in fact $f=0$ globally (i.e. on the entirety of $V^{\otimes k}$).   We do this iteratively, proving by the following $t$-dependent claim by induction on $t$:

{\bf Claim($t$):} The restriction of $f$ to any summand of the type
\begin{equation}\label{eq:vis}
  V_{i_1}\otimes \cdots \otimes V_{i_k},\ t\text{ of the }i_j\text{ are }2.  
\end{equation}
is zero.

The base case $t=0$ of the induction is in place (claiming simply that $f|_{V_1^{\otimes k}}=0$, which we know). We now turn to the induction step, assuming Claim($s$) for all $s\le t-1$ and seeking to prove Claim($t$). In order to lessen the notational load of the argument we will focus on $V_1^{\otimes (k-t)}\otimes V_2^{\otimes t}$ (i.e. we assume the {\it last} $t$ indices in \Cref{eq:vis} are $2$).

Equivalently, we have to show that $f$ is zero on
\begin{equation*}
  V_1^{\otimes (k-t)}\otimes \bC e_j\otimes V_2^{\otimes (t-1)} 
\end{equation*}
for any $j\in [M+1,M+T]$. To do this, first note that because $V_2$ is self-dual over $S^+_M* S^+_T$, hence also over its quantum subgroup $\bG$, the $\bG$-morphisms $V_1^{\otimes (k-t)}\otimes V_2^{\otimes t}\to \bC$ are in bijection with the $\bG$-morphisms
\begin{equation*}
  V_1^{\otimes (k-t)}\to V_2^{\otimes t}.
\end{equation*}
Since moreover $V_1$ is the induction of the corresponding $S^+_M$-representation, the adjunction \Cref{eq:frob} reads
\begin{equation}\label{eq:frob2}
  \mathrm{hom}_{\bG}\left(V_1^{\otimes (k-t)},V_2^{\otimes t}\right)\cong \mathrm{hom}_{S^+_M}\left(V_1^{\otimes (k-t)},\mathrm{ind} V_2^{\otimes t}\right)
\end{equation}
Now, the finite-dimensional $S^+_M$-representation $\mathrm{ind} V_2^{\otimes t}$ embeds into some tensor power $V_1^{\otimes s}$, and hence elements in the right hand side of \Cref{eq:frob2} are spanned by non-crossing partitions of $[s+k-t]$. Since non-crossing partitions are linearly independent on spaces of dimension $\ge 4$, morphisms in \Cref{eq:frob2} vanish if they do so when restricted to
\begin{equation*}
  (V_1')^{\otimes (k-t)},\ V_1' = \mathrm{span}\{e_1,\cdots,e_{M-1}\}. 
\end{equation*}
In our setting, what this means is that it is enough to prove that $f$ vanishes on
\begin{equation}\label{eq:otimes-target}
  (V_1')^{\otimes(k-t)}\otimes \bC e_j\otimes V_2^{\otimes (t-1)}. 
\end{equation}

Set
\begin{equation*}
  V_2' = \mathrm{span}\{e_\ell\ |\ \ell\in [M+1,M+T]-\{j\}\}. 
\end{equation*}
Then, $(V_1')^{\otimes(k-t)}\otimes \bC e_j\otimes (V_2')^{\otimes (t-1)}$ is contained in the image of $V_1^{\otimes (k-t+1)}\otimes V_2^{\otimes (t-1)}$ through a the permutation of the $e_i$ interchanging $e_M$ and $e_j$. Since $f$ is an $S_N$-morphism, Claim($t-1$) now implies that $f$ vanishes on
\begin{equation*}
(V_1')^{\otimes(k-t)}\otimes \bC e_j\otimes (V_2')^{\otimes (t-1)}.
\end{equation*}
Similarly, $f$ vanishes on the other summands of \Cref{eq:otimes-target} resulting from the decomposition $V_2=V'_2\oplus \bC e_j$ by the other instances Claim($s$), $s\le t-2$ and this concludes the proof.
\end{proof}

\subsection{Quantum reflection groups}\label{subse.hyp}

We now turn to the quantum reflection groups $H_N^{s+}$ for $1\le s\le\infty$. The main result of the subsection is

\begin{theorem}\label{th.hns}
  For $N\ge 4$ and $1\le s\le \infty$ the dual $\widehat{H_N^{s+}}$ is residually finite and hence also hyperlinear. 
\end{theorem}
\begin{proof}
We fix $N$ throughout, and denote by $\cA_s$ and $\cA$ the Hopf $*$-algebras associated to $H_N^{s+}$ and $S^+_N$ respectively. 

{\bf (Case 1: $s<\infty$)} Recall from \cite[Theorem 3.4 (2)]{bv-reflection} that we have a free wreath product decomposition
\begin{equation}\label{eq:wr}
  \cA_s\cong C(\bZ_s)*_w \cA,
\end{equation}
where the right hand side is by definition the free $*$-algebra generated by $\cA$ and $n$ copies of $C(\bZ_s)$ with the constraint that the $i^{th}$ copy of $C(\bZ_s)$ commutes with the $i^{th}$ row of generators $u_{ij}$, $1\le j\le N$ of $\cA$. It follows that $\cA_s$ can be realized as a succession of extensions of the form
\begin{equation*}
  \cB\mapsto C(\bZ_s)*\cB/[C(\bZ_s),\cD]
\end{equation*}
for a commutative finite-dimensional $*$-subalgebra $\cD\subset \cB$. More concretely, the various algebras $\cD$ are the $N$-dimensional subalgebras of $\cA$ generated by a row $u_{ij}$, $1\le j\le N$ of generators. 

The residual finite-dimensionality of $\cA_s$ then follows inductively from \Cref{pr.push-rf}. 

{\bf (Case 2: $s=\infty$)} Given that $H^{\infty +}_N$ is topologically generated by all of the finite $H^{s+}_N$ embedded therein (by \Cref{le.hn-top-gen}), the conclusion follows from \Cref{le.gen-rf}. 
\end{proof}

While the proof given above for \Cref{th.hns} does not proceed inductively on $N$ or require topological generation, there is nevertheless an analogue of \Cref{th.top-gen} for quantum reflection groups that we record here. Just as we did for quantum permutation groups, we regard $H^{s+}_{N-1}$ as a quantum subgroup of $H^{s+}_{N}$ via the map
\begin{equation*}
  \cA(H^{s+}_N)\to \cA(H^{s+}_{N-1}) 
\end{equation*}
that sends the generators $u_{ij}$ to $\delta_{ij}$ if $N\in \{i,j\}$. 

\begin{theorem}\label{th.refl-tg}
  For $N\ge 6$ and $1\le s\le\infty$ the quantum reflection group $H_N^{s+}$ is topologically generated by its quantum subgroups $H^{s+}_{N-1}$ and $S_N$. 
\end{theorem}
\begin{proof}
We once more separate the finite and infinite-$s$ cases.
  
{\bf (1: finite $s$)} This follows from a slight adaptation of \Cref{pr.coinv}, modified as follows based on the representation theory of $H_N^{s+}$ as developed in \cite[Sections 5-7]{bv-reflection}. Instead of a single $n$-dimensional fundamental representation $V$ we have $s$ of them, labeled $V^{(i)}$ for $i\in \bZ_s$. As for morphisms,
\begin{equation*}
  f:V^{(i_1)}\otimes \cdots \otimes V^{(i_k)}\to \bC
\end{equation*}
is $H^{s+}_{N}$-invariant precisely when it is a linear combination of non-crossing partitions whose blocks are of the form
\begin{equation*}
  \{i_{a_1},\ \cdots ,\ i_{a_t}\},\ \sum_j i_{a_j} = 0\in \bZ_s. 
\end{equation*}
The argument in the proof of \Cref{pr.coinv} then goes through virtually unchanged.

{\bf (2: $s=\infty$)} This follows as in the proof of \Cref{th.hns}, from the claim for $s<\infty$ and the fact that $H^{s+}_N$ topologically generate $H^{\infty +}_{N}$ (by \Cref{le.hn-top-gen}). 
\end{proof}

\begin{remark}
  Note the bound on $N$: unlike \Cref{th.top-gen}, \Cref{th.refl-tg} does not apply to $N=5$ hence does not yield an inductive proof of residual finite-dimensionality; we do not know whether the result is still valid in that case, but believe it to be. 
\end{remark}

\subsection{Remarks on free entropy dimension} \label{section:FED} In this section we make some brief remarks on what is currently known about the {\it free entropy dimension} of the canonical generators of the finite von Neumann algebras $L^\infty(H_N^{s+})$.  We refer the reader to the survey \cite{Vo02} and the references therein for details on the various versions of free entropy dimension that exist.  

Fix $N \in \bN$ and  $s \in \bN \cup \{\infty\}$ and consider the self-adjoint family $X(N,s) = \{u_{ij}, u_{ij}^*\}_{1 \le i,j \le N}$ of generators of $\cA (H_N^{s+}) \subset L^\infty(H_N^{s+})$.  Associated to the sets $X(N,s)$ we have the (modified) {\it microstates free entropy  dimension} $\delta_0(X(N,s)) \in [0, n]$ and the {\it non-microstates free entropy dimension} $\delta^*(X(N,s))\in [0, n]$.  From \cite{BiCaGu} it is known that the general inquality $\delta_0(\cdot) \le \delta^*(\cdot)$ always holds,  and from \cite{CoSh} an upper bound for $\delta^*$ exists in terms of the $L^2$-Betti numbers of the discrete dual quantum groups $\widehat{H_N^{s+}}$: 
\[
\delta^*(X(N,s)) \le \beta_1^{(2)}(\widehat{H_N^{s+}}) - \beta_0^{(2)}(\widehat{H_N^{s+}}) +1.
\]
Here $\beta_k^{(2)}(\cdot)$ is the {\it $k$th $L^2$-Betti number} of a discerete quantum group.  See for example \cite{Ve12, KyRaVaVa, Bi13}.  Now, in \cite[Theorem 5.2]{KyRaVaVa}, we have the following computations

\[
\beta_1^{(2)}(\widehat{H_N^{s+}}) = 1-\frac 1s \quad \&\quad  \beta_0^{(2)}(\widehat{H_N^{s+}}) = 0 \qquad (N \ge 4).
\]
Finally, since $L^\infty(H_N^{s+})$ is Connes embeddable by Theorem \ref{th.hns}, it follows from \cite[Corollary 4.7]{Ju03} that $\delta_0(X(N,s)) \ge 1$ whenever $L^\infty(H_N^{s+})$ is diffuse.  The question of when exactly $L^\infty(H_N^{s+})$ is diffuse still seems to be open in complete generality.  However, it is known that $L^\infty(H_N^{s+})$ is a II$_1$-factor (and in particular diffuse) $N \ge 8$ \cite{Br13, lemeux2013fusion, Wahl15}.    Combing all the above inequalities together, we finally obtain 

\begin{corollary}\label{cor:FED}
For $N \ge 8$ and $s \in \bN \cup \{\infty\}$, we have 
\[
1 \le \delta_0(X(N,s)) \le \delta^*(X(N,s)) \le 2-\frac1s.
\]
In particular, the generators $X(N) = \{u_{ij}\}_{1 \le i,j \le N}$ of  $L^\infty(S_N^+)$ satisfy $\delta_0(X(N)) = \delta^*(X(N)) = 1$ for $N \ge 8$.
\end{corollary}

\begin{remark} \label{rem:strong-1bounded}
The situation for $S_N^+$ in Corollary \ref{cor:FED} is similar to what happens for the free orthogonal quantum groups $O_N^+$ \cite{bcv}.  However, for $O_N^+$, even more is known:  In \cite{BrVe18} it was shown that in fact $L^\infty(O_N^+)$ is a {\it strongly 1-bounded} von Neumann algebra for all $N \ge 3$.  The notion of strong 1-boundedness was introduced by Jung in \cite{Ju07} and entails that $\delta_0(X) \le 1$ for any self-adjoint generating set $X \subset L^\infty(O_N^+)$.  In particular, it follows that $L^\infty(O_N^+)$ is never isomorphic to an interpolated free group factor.  In this context, it is natural to ask whether similar results hold for quantum permuation groups: \\ \\
{\it Is $L^\infty(S_N^+)$ a strongly 1-bounded von Neumann algebra for all $N \ge 5$?}
\end{remark}

\section{Inner faithful matrix models for quantum permutation groups} \label{se.mod}

\Cref{th.rf} shows that the quantum group algebras $\cA(S_N^+) = \bC \widehat{S^+_N}$ have enough finite-dimensional $*$-representations, i.e. map faithfully into a product of matrix algebras. In the present section we prove that a specific, canonical collection of ``elementary'' representations is faithful in a certain sense. First, let us clarify the appropriate notion of faithfulness here.  See e.g. \cite[Definition 2.7]{bb-inner} for more details.

\begin{definition}\label{def.inner}
A $\ast$-homomorphism $\pi:\cA \to B$ from a Hopf $\ast$-algebra $\cA$ into a $\ast$-algebra $B$ is {\it inner faithful} if $\ker \pi$ contains no non-trivial Hopf $\ast$-ideals. Equivalently, for any factorization $\pi =
  \tilde \pi \circ \rho$ with $\rho : \cA \to \cA'$, a surjective morphism of Hopf $\ast$-algebras, we have in fact that $\rho$ is an isomorphism.  More generally, the {\it Hopf image} of $\pi:\cA \to B$ is the ``smallest'' quotient Hopf $\ast$-algeba $\cA'$ such that  $\pi$ factors through the quotient map $\rho : \cA \to \cA'$.  
 \end{definition}

Note that the Hopf image always exists and is unique (up to isomorphism) \cite{bb-inner}.  In this paper, we will only be concerned with the cases where our Hopf $\ast$-algebras are of the form $\cA(\bG)$ for a compact quantum group $\bG$.  In this case, the Hopf image of $\pi:\cA(\bG) \to B$ is $\cA(\bH)$, where $\bH < \bG$ is the ``smallest'' quantum subgroup  such that $\pi$ factors through $\rho : \cA(\bG) \to \cA(\bH)$. Moreover, $\pi$ is inner faithful if and only if $\bH = \bG$ (up to isomorphism).     

\subsection{Flat matrix models}
Following \cite[Definition 5.1]{bn-flat}, we denote by $X_N$ the space of $N\times N$ matrices $P = (P_{ij})_{ij}  \in M_N(M_N(\bC))$, whose entries $P_{ij} \in M_N(\bC)$, are rank-one projections with the property that for all $1\le i,j\le N$,
 \begin{equation*}
  \sum_j P_{ij} = 1 = \sum_i P_{ij}
\end{equation*}
In other words, $X_N \subset M_N(M_N(\bC))$ is the compact set of {\it bistochastic} $N\times N$ matrices of rank one projections in $M_N$. 

It is clear that each $P \in X_N$ gives rise to a $\ast$-homomorphism \begin{eqnarray*}
\pi_P:\cA(S_N^+) \to M_N(\bC); \qquad \pi_P(u_{ij}) = P_{ij} \qquad (1 \le i,j \le N),
\end{eqnarray*} 
and we call $\pi_P$ a {\it flat matrix model} for the quantum group $S_N^+$.  If we package all these flat matrix models $\pi_P$ into one single representation by allowing $P \in X_N$ to vary, we arrive at a construction that features prominently in \cite{bn-flat,bf-model}.

\begin{definition}\label{def.univ}
  The {\it universal flat matrix model} of $S^+_N$ is the morphism
  \begin{equation}\label{eq:pi}
    \pi: \cA(S^+_N)\to M_N(C(X_N))\cong C(X_N,M_N(\bC)); \qquad \pi(u_{ij}) = \{P \mapsto \pi_P(u_{ij}) = P_{ij}\}.
  \end{equation}
\end{definition}


One of the main conjectures in \cite{bn-flat} is the inner faithfulness of the universal flat matrix models.

\begin{conjecture}\cite[Conjecture 5.7]{bn-flat}\label{cj.if}
  The universal flat matrix model $\cA(S^+_N)\to M_N(C(X_N))$ is faithful for $N=4$ and inner faithful for $N\ge 5$.  
\end{conjecture}

\begin{remark}\label{re.cj-4}
  Note that for $N=4$ the conjecture is true by \cite[Theorem 4.1]{bc-pauli} (or \cite[Proposition 2.1]{bn-flat}). Indeed, the latter result produces {\it some} faithful representation of the form
  \begin{equation*}
    \cA(S^+_4)\to M_4(C(X))
  \end{equation*}
  for compact $X$ (in fact $X=SU_2$), which must factor as
  \begin{equation*}
    \cA(S^+_4)\to M_4(C(X_4))\to M_4(C(X))
  \end{equation*}
  by universality. 
\end{remark}

The main result of this section is that this conjecture is true for at least almost all $N$. To begin, we first need a few remarks and observations. Consider the closed subspace $X_{N}^{class}\subset X_{N}$ of matrices $P \in X_N$ for which the entries $P_{ij}$ pairwise commute.  Then, the classical permutation group $S_{N}$ also has a universal flat matrix model $\pi^{class} : \cA(S_{N})\to M_{N}(C(X_{N}^{class}))$. Moreover, if $q : \cA(S_{N}^{+})\to \cA(S_{N})$ is the canonical quotient map and $r : C(X_{N})\to C(X_{N}^{class})$ is the restriction map, then by construction $\pi^{class}\circ q= r\circ\pi$.

\begin{lemma}\label{le.not-cls}
  If $N\ge 4$ then the inclusion $X_N^{class}\subset X_N$ is proper.   
\end{lemma}
\begin{proof}
To produce examples of $N\times N$ bistochastic matrices whose entries do not all commute we proceed as follows.

First, fix a basis $e_i$, $1\le i\le N$ for $\bC^N$ and let $L \in M_N(\bN)$ be a Latin square of size $N\times N$ (meaning that each row and column is a permutation of $\{1,\cdots,N\}$). Assume furthermore that the upper left hand $2\times 2$ corner of $L$ is
\begin{equation}\label{eq:22}
  \begin{pmatrix}
    1&2\\
    2&1
  \end{pmatrix}
\end{equation}
We can then form the bistochastic and commutative matrix whose $(i,j)$ entry is the projection on the one-dimensional span of $e_{L_{ij}}$, and then modify it slightly by changing its upper left hand $2\times 2$ corner to
\begin{equation*}
  \begin{pmatrix}
    P_u & P_v\\
    P_v & P_u
  \end{pmatrix}
\end{equation*}
where $u=e_1+e_2$ and $v=e_1-e_2$. The resulting bistochastic matrix contains, say, the projections $P_u$ and $P_{e_1}$, which do not commute.

It remains to argue that a Latin square $L$ as above exists if $N\ge 4$ (we have not used this hypothesis thus far). To see this observe that for $N\ge 4$ we can complete the square \Cref{eq:22} to a $2\times n$ {\it Latin rectangle}, in the sense that the two rows are permutations of $\{1,\cdots,N\}$ and no two entries in the same column coincide. We can then use the result that any Latin rectangle can be completed to a Latin square (e.g. \cite[Chapter 35, Lemma 1]{az}).
\end{proof}

This already yields one instance of  \Cref{cj.if}.

\begin{proposition}\label{le.N=5}
$S^+_5$ satisfies \Cref{cj.if}. 
\end{proposition}
\begin{proof}
Let $\pi, \pi^{class}, r, q$ be as above and assume $N=5$.   Let $I$ be a Hopf $*$-ideal contained in $\ker(\pi)$ and observe that $\pi^{class}(q(I)) = r\circ\pi(I) = 0$. Since $\pi^{class}$ is inner faithful (see e.g., \cite[Proposition 5.3]{bf-model}), this forces $q(I) = (0)$, i.e.~$I\subset \ker(q)$. But because there is no intermediate quantum group between $S_5$ and $S_5^+$ by \cite[Theorem 7.10]{ban-uni}, it follows that either $I = (0)$ or $I = \ker(q)$. In the second case we get
\begin{equation*}
\pi = \pi^{class}\circ q = r\circ\pi.
\end{equation*}
In particular, this implies that any family $(P_{ij})_{1\leqslant i, j\leqslant N}$ of rank-one projections which are pairwise orthogonal on rows and columns commutes, i.e.~$X_{N} = X_{N}^{class}$. This equality, however, is invalid for $N\geqslant 4$ by \Cref{le.not-cls}.
\end{proof}

\begin{remark}\label{re.max-if}
  In fact, the proof of \Cref{le.N=5} shows that \Cref{cj.if} is satisfied whenever the inclusion $S_N < S^+_N$ is maximal. 
\end{remark}

We are now ready for the main technical result of this section.

\begin{proposition}\label{pr.cj-ind}
Let $M \ge 5$ and $N\ge 2M$. If $S^+_M$ satisfies \Cref{cj.if} then so does $S^+_N$.   
\end{proposition}
\begin{proof}
Let $\cA=\cA(S^+_N)$. We have to argue that under the hypothesis, the Hopf image $\cA\to \cA_{\pi}$ of \Cref{eq:pi} is all of $\cA$. Since the quantum group attached to $\cA_{\pi}$ clearly contains $S_N$, \Cref{pr.gen-diag} reduces the problem to showing that it also contains a quantum subgroup $\bG<S^+_N$ that is $(M,N)$-large in the sense of \Cref{def.large}.

For this, fix a collection of rank-one projections $P_i$, $0\le i\le N-1$ in $M_N$ summing up to $1$. We form an $N\times N$ Latin square $(\cL_{ij})_{0\le i,j\le N-1}$ all of whose entries are the projections $P_i$ as follows:
\begin{itemize}
\item if $i,j\le M-1$ we set $\cL_{ij}=P_{(i-j)\;\mathrm{mod}\;M}$;
\item we fill the rest of the first $M$ rows with $P_i$'s arbitrarily so as to retain the Latin rectangle property (this is possible because $2M\le N$); 
\item complete the above Latin rectangle to a Latin square, once more using \cite[Chapter 35, Lemma 1]{az}.    
\end{itemize}
Having constructed $\cL$ as above, consider the subspace $Y\subset X_N$ consisting of those $N\times N$ bistochastic matrices $\cM$ of rank-one projections that are identical to $\cL$ outside of the upper left hand $M\times M$ corner.

Setting
\begin{equation*}
  P=\sum_{i=0}^{M-1} P_i,
\end{equation*}
all operators appearing as entries of matrices $\cM\in Y$ commute with $P$. Restricting these operators to the range of $P$ (which is in turn isomorphic to $\bC^M$), we obtain the upper right hand arrow in the composition
\begin{equation*}
  \begin{tikzpicture}[baseline=(current  bounding  box.center),anchor=base,cross line/.style={preaction={draw=white,-,line width=6pt}}]
    \path (0,0) node (1) {$\cA(S^+_N)$} +(3,.5) node (2) {$M_N(C(X_N))$} +(6,.5) node (3) {$M_N(C(Y))$} +(9,0) node (4) {$M_M(C(Y))$} +(4.5,-.5) node (d) {$\cA(S^+_M)*\cA(S^+_{N-M})$}; 
    \draw[->] (1) to[bend left=6]  node[pos=.5,auto]{$\scriptstyle \pi$} (2);
    \draw[->] (2) to[bend left=6] (3);
    \draw[->] (3) to[bend left=6] (4);
    \draw[->] (1) to[bend right=6] (d);
    \draw[->] (d) to[bend right=6] node[pos=.5,auto,swap] {$\scriptstyle \eta$} (4);    
  \end{tikzpicture}  
\end{equation*}
where the lower factorization occurs because by construction the off-block-diagonal entries $\cL_{ij}$ with precisely one of $i,j$ in $\{0,\cdots,M-1\}$ are projections orthogonal to $P$ and hence vanish on $\mathrm{Im}\;P$. Our goal is now to show that the Hopf image of $\eta$ in the above diagram contains an $(M,N)$-large quantum subgroup
\begin{equation*}
  \bG < S^+_M * S^+_T < S^+_N,\ T:=N-M. 
\end{equation*}
Equivalently, this means proving that the composition
\begin{equation}\label{eq:eta-comp}
  \begin{tikzpicture}[baseline=(current  bounding  box.center),anchor=base,cross line/.style={preaction={draw=white,-,line width=6pt}}]
    \path (0,0) node (1) {$\cA(S^+_M)$} +(3,0) node (2) {$\cA(S^+_M)*\cA(S^+_T)$} +(6,0) node (3) {$M_M(C(Y))$}; 
    \draw[->] (1) -- (2);
    \draw[->] (2) --(3)node[pos=.5,auto]{$\scriptstyle \eta$};
  \end{tikzpicture}
\end{equation}
is inner faithful. To verify this, recall that by construction the upper left hand $M\times M$ corners of matrices in $Y$ are {\it arbitrary} bistochastic matrices in
\begin{equation*}
  M_M\cong \mathrm{End}(\mathrm{Im}\;P). 
\end{equation*}
Since the other entries of matrices in $Y$ are identical to those of the fixed Latin square $\cL$, we have an isomorphism $M_M(C(Y))\cong M_M(C(X_M))$. The isomorphism described here renders \Cref{eq:eta-comp} identical to the canonical universal flat representation of $\cA(S^+_M)$, which is inner faithful by hypothesis.
\end{proof}

As a consequence, we can prove \Cref{cj.if} for almost all $N$.

\begin{corollary}\label{cor.div5}
All $S^+_N$ with $N \le 5$ and $N\ge 10$ satisfy \Cref{cj.if}. 
\end{corollary}
\begin{proof}
As explained above, we already know the conjecture to hold in the cases $N \le 4$.  For the remaining cases , it is enough to prove that $S_{5}^{+}$ satisfies \Cref{cj.if} in view of \Cref{pr.cj-ind}, and this is taken care of by \Cref{le.N=5}.
\end{proof}

\subsection{Inner unitary Hopf $\ast$-algebras}\label{subse.inner-unitary} 

\Cref{cor.div5} shows that the universal flat matrix model is inner faithful for most quantum permutation groups.  In this final section we show that we can do even better: it turns out that for the same values of the parameter $N$ a {\it single} finite-dimensional representation suffices to achieve inner faithfulness. We first recall the relevant concept from \cite[Definition 5.1]{ab}. 

\begin{definition}\label{def.inner-unitary}
A Hopf $*$-algebra $\cA$ is {\it inner unitary} if it has an inner faithful $*$-homomorphism into a finite-dimensional C$^*$-algebra.  
\end{definition}

The main result of this subsection is the following improvement on \Cref{cor.div5}.  

\begin{theorem}\label{th.inner-unitary}
  The Hopf $\ast$-algebra $\cA=\cA(S^+_N)$ is inner unitary for all $N$ outside the range $[6,9]$.
\end{theorem}

\begin{proof}
  We first tackle the smaller-$N$ cases.

  {\bf (Case 1: $N\le 3$)} $S^+_N$ is classical, and hence the conclusion follows from \cite[Proposition 5.5]{bb-inner}.  
  
  {\bf (Case 2: $N=4,5$)} Let $x\in X_N$ be any of the bistochastic matrices whose entries generate a non-commutative subalgebra of $M_N(\bC)$ and let $y\in X_{N}^{class}\subset X_N$ be such that the corresponding flat representation $\pi^{class}_y$ is inner faithful on $\cA(S_N)$ (such $x$ and $y$ exist by \Cref{le.not-cls}) and \cite[Proposition 5.3]{bf-model}, respectively). The Hopf image of the  representation
  \begin{equation*}
    \pi_x\oplus\pi_y:\cA(S^+_N)\to M_N(\bC)\oplus M_N(\bC)
  \end{equation*}
is then a non-commutative quotient Hopf $*$-algebra of $\cA$ containing $S_N$. Since for $N=4,5$ we know that there are no intermediate quantum groups
  \begin{equation*}
    S_N<\bG<S^+_N,
  \end{equation*}
the Hopf image of $\pi_x\oplus\pi_y$ is all of $\cA$, as desired.

{\bf (Case 3: $N\ge 10$)} The proof of \Cref{pr.cj-ind} in fact shows that if $5\le M\le \frac N2$ and $\cA(S^+_M)$ admits a finite inner faithful family $\{\pi_{z_1}, \ldots, \pi_{z_n}\}$ ($z_i \in X_M$) of flat $M$-dimensional representations, then $\cA$ admits a finite family $\{\pi_{z_1'}, \ldots, \pi_{z_n'}\}$ ($z_i' \in X_N$) of flat representations whose joint Hopf image surjects onto $\cA(S^+_M)*\cA(S^+_{N-T})$. Since for $M=5$ we do have such a family $\{\pi_x,\pi_y\}$ by the previous step of the current proof, we have such an $x', y'\in X_N$.


Further choosing any inner faithful flat representation $\pi^{class}_w:\cA(S_N)\to M_N(\bC)$  ($w \in X_N^{class}$), the resulting direct sum representation \[\pi_{x'} \oplus \pi_{y'} \oplus \pi_{w}: \cA(S^+_N)\to M_N(\bC)^{\oplus 3}\] is inner faithful. Indeed, the quantum subgroup of $S^+_N$ dual to its Hopf image contains both $S_N$ and $S^+_M * S^+_{N-M}$ and hence coincides with $S^+_N$ by \Cref{cor.top-gen-bis}.
\end{proof}


\bibliography{RFD_QG}{}
\bibliographystyle{plain}
\addcontentsline{toc}{section}{References}

\Addresses

\end{document}